\documentclass[11pt,a4paper]{article}
\usepackage{a4wide}
\usepackage{amssymb,amsfonts,latexsym,stmaryrd}
\usepackage{QED}
\usepackage[PostScript=dvips]{diagrams}
\usepackage{euscript}
\usepackage{epsfig}
\usepackage{tikz,ifdraft}
\usetikzlibrary{arrows,positioning,matrix,backgrounds,calc,fit}

\newtheorem{theorem}{Theorem}[section]

\newtheorem{proposition}[theorem]{Proposition}

\newcommand{\beqa}{\begin{eqnarray*}}
\newcommand{\eeqa}{\end{eqnarray*}\par\noindent}

\newcommand{\id}{\mathsf{id}}
\renewcommand{\emph}[1]{\textbf{#1}}

\newcommand{\PP}{\mathbf{P}}

\newcommand{\lrarr}{\longrightarrow}
\newcommand{\rarr}{\rightarrow}

\newcommand{\ie}{\textit{i.e.}~}
\newcommand{\CC}{\mathcal{C}}

\newcommand{\Set}{\mathbf{Set}}
\newcommand{\Two}{\mathbf{2}}

\newcommand{\IFF}{\; \Longleftrightarrow \;}
\newcommand{\AND}{\; \wedge \;}
\newcommand{\OR}{\; \vee \;}
\newcommand{\IMP}{\; \Rightarrow \;}

\newcommand{\card}[1]{|#1|}

\newarrow{Eq}=====
\newarrow{mon}>--->
\newarrow{mto}|--->
\newarrow{rel}--+->
\newarrow{inc}C--->

\diagramstyle[newarrowhead=vee,newarrowtail=vee]

\newcommand{\rinj}{\rmon[heads=vee]}

\newcommand{\op}{\mathsf{op}}
\newcommand{\natarrow}{\rTo^{.}}

\newcommand{\UU}{\mathcal{U}}

\newcommand{\Pref}{\mathbb{P}}
\newcommand{\sw}{\sigma}
\newcommand{\swl}{\sigma^{L}}
\newcommand{\Cinc}{\CC_{\mathsf{inc}}}
\newcommand{\Pinc}{\Pref_{\mathsf{inc}}}
\newcommand{\Ind}{\mathcal{I}}

\newcommand{\IIA}{\mbox{\textbf{\textrm{IIA}}}}
\newcommand{\Pareto}{\mbox{\textbf{\textrm{P}}}}
\newcommand{\CP}{\mbox{\textbf{\textrm{CP}}}}
\newcommand{\vsa}{\vspace{0.1in}}

\newcommand{\UD}{\mbox{\textbf{\textrm{UD}}}}
\newcommand{\CUD}{\mbox{\textbf{\textrm{CUD}}}}
\newcommand{\WP}{\mbox{\textbf{\textrm{WP}}}}
\newcommand{\Dict}{\mbox{\textbf{\textrm{D}}}}
\newcommand{\Dinc}{\mathcal{D}_{\mathsf{inc}}}
\newcommand{\Dom}{\mathcal{D}}
\newcommand{\DL}{\mathcal{D}^{L}}
\newcommand{\DLinc}{\mathcal{D}^{L}_{\mathsf{inc}}}
\newcommand{\Dinj}{\mathcal{D}}
\newcommand{\pstr}{p^{>}}
\newcommand{\str}[1]{#1^{>}}
\newcommand{\sigstr}{\str{\sigma(p)}}
\newcommand{\UC}{U^{\mathsf{c}}}

\newcommand{\Lin}{\mathbb{L}}
\newcommand{\Lininc}{\mathbb{L}_{\mathsf{inc}}}
\newcommand{\LL}{\mathcal{L}}
\newcommand{\PI}{\Pinc^{\Ind}}
\newcommand{\AU}{\mathbb{A}}
\newcommand{\TI}{\Two^{\Ind}}

\title{Arrow's Theorem by Arrow Theory}

\author{Samson Abramsky\\
Department of Computer Science, University of Oxford}

\date{}

\begin{document}

\maketitle

\begin{abstract}
We give a categorical account of Arrow's theorem, a seminal result in social choice theory.
\end{abstract}

\section{Introduction}

Arrow's theorem \cite{arrow1950difficulty,arrow1970social} is a seminal result from 1950 which founded the modern field of \emph{social choice theory}.
It is a no-go theorem for combining individual preferences: under two very plausible assumptions, the only possible `social welfare function' is a dictatorship!
Closely allied results, notably the Gibbard-Satterthwaite theorem \cite{gibbard1973manipulation,satterthwaite1975strategy}, relate to \emph{voting systems}: when there are at least three alternatives, the only non-manipulable voting systems are dictatorships.

Our aim in this note is  to cast this result at a more abstract, categorical level, to  expose common structure with no-go theorems in other fields, such as the foundations of quantum mechanics.
This is motivated by our recent work with Adam Brandeburger on a 
sheaf-theoretic approach to quantum non-locality and contextuality \cite{abramsky2011sheaf},
and also by the work of Jouko V\"a\"an\"anen and others on Dependence/Independence Logic \cite{vaananen2007dependence,gradeldependence}.

We shall assume only the most basic notions of category theory: categories, functors, and natural transformations. Moreover, these will be used in a simple, concrete setting. Thus the categorical language should not be an impediment to readers interested in a structural account of Arrow's theorem.

\subsection{What Arrow's Theorem Says}

Let $A$ be a set of \emph{alternatives}, and $\Ind$ a set of \emph{individuals}.

$\Pref(A)$ is the set of \emph{preference relations} on $A$. These are usually taken to be weak orders (transitive and connected relations), sometimes linear orders (transitive, connected and antisymmetric).

$\Pref(A)^{\Ind}$ is the set of \emph{profiles} or \emph{ballots}, which assign a preference relation on the alternatives to each individual --- a ``vote''.

A \emph{social welfare function} is a map
\[ \sw : \Pref(A)^{\Ind} \lrarr \Pref(A) . \]
Such a map produces a single ranking on alternatives --- a social choice --- from a profile.

\vsa
Two conditions are standardly considered on such functions:
\begin{itemize}
\item \textbf{Independence of Irrelevant Alternatives} ($\IIA$). The social decision on the relative preference between two alternatives $a$, $b$ depends only on the individual preferences between these alternatives. It is \emph{independent} of their rankings with respect to other alternatives.

\item \textbf{The Pareto or Uniformity Principle} ($\Pareto$). If every individual prefers $a$ to $b$, then so should the social welfare function.
\end{itemize}

We can now state Arrow's Theorem more formally as follows.

\begin{theorem}[Arrow's Theorem]
If $\card{A} > 2$ and $\Ind$ is finite, then any social welfare function satisfying $\IIA$ and $\Pareto$ is a \emph{dictatorship}: \ie for some individual $i \in \Ind$, for all profiles $p \in \Pref(A)^{\Ind}$ and alternatives $a, b \in A$:
\[  a \, \sw(p) \, b \IFF   a \, p_i \, b . \]
\end{theorem}

Thus the social choice function, under these very plausible assumptions, simply copies the choices of one fixed individual --- the dictator.

An extraordinary number of different proofs, as well as innumerable variations, have appeared in the (huge) literature. For a small selection, see \cite{arrow1970social,blau1972direct,kirman2012arrow,taylor2005social}.

A closely related result is the \emph{Gibbard-Satterthwaite theorem} \cite{gibbard1973manipulation,satterthwaite1975strategy} on voting systems:

\begin{theorem}
If $\card{A} > 2$ and $\Ind$ is finite, then any voting system
\[ v : \Pref(A)^{\Ind} \lrarr A \]
which is \emph{non-manipulable} is a dictatorship.
\end{theorem}

The following quotation from the recent text \cite{taylor2005social} nicely captures the significance of the result:
\begin{quote}
``For an area of study to become a recognized field, or even a recognized subfield, two things are required: It must be seen to have coherence, and it must be seen to have depth. The former often comes gradually, but the latter can arise in a single flash of brilliance.
\ldots With social choice theory, there is little doubt as to the seminal result that made it a recognized field of study: Arrow's impossibility theorem.''
\end{quote}

The further contents of the paper are as follows. In Section~2, we shall present a fairly standard account of Arrow's theorem which will fix notation and serve as a reference point. In Section~3, we will reformulate  Arrow's Theorem in categorical terms, and in Section~4, we shall give a development of the proof which uses the categorical formulation to emphasize the structural aspects. Section~5 concludes.

\section{A `standard' account of Arrow's theorem}
\label{concsec}

The aim of this section is to give a clear, explicit presentation of a fairly standard account of Arrow's theorem and some related notions.

The arguments in Section~2.1 follow \cite{arrow1970social}, with some clarifications and refinements due to \cite{blau1972direct}. In Section~2.2, we follow \cite{kirman2012arrow}.

\subsection{Preference Relations}

We consider a number of properties of binary relations $R \subseteq A^2$ on a set $A$.
These are all universally quantified over elements $a, b, c \in A$:
\begin{description}
\item[Reflexivity]  $aRa$
\item[Irreflexivity] $\neg aRa$
\item[Symmetry] $aRb \IMP bRa$
\item[Antisymmetry] $aRb \AND bRa \IMP a=b$
\item[Transitivity] $aRb \AND bRc \IMP aRc$
\item[Connectedness] $aRb \OR bRa$.
\end{description}

A \emph{weak preference relation} is a transitive connected relation. We write $\Pref(A)$ for the set of all weak preference relations on $A$. Given a weak preference relation $R$, we can define two other relations:
\begin{description}
\item[Strict Preference]  $a P b := aRb \AND \neg bRa$.
\item[Indifference] $aIb := aRb \AND bRa$.
\end{description}
Then $P$ is a strict ordering (transitive and irreflexive), while $I$ is an equivalence relation (reflexive, symmetric and transitive).
These relations satisfy the following properties:
\begin{description}
\item[Trichotomy] $aPb \OR bPa \OR aIb$.
\item[Absorption] $aIb \AND bPc \AND cId \IMP aPd$.
\end{description}

A weak preference relation is \emph{linear} if it additionally satisfies antisymmetry. We write $\Lin(A)$ for the set of linear preference relations on $A$. If $R$ is linear, then the associated indifference relation $I$ is just the identity relation, while $P$ is a strict linear order.

Given $A \subseteq B$, we can define a restriction map $\Pref(B) \rarr \Pref(A) :: R \mapsto R | A$, where $R | A := R \cap A^2$. Note that the truth of any property of $R$ expressed by a universal sentence is preserved under restriction, so this is well-defined; moreover, the same operation also defines a map $\Lin(B) \rarr \Lin(A)$.

\subsection{Social choice situations}

We shall define a class of structures which provide the setting for Arrow's theorem.

A \emph{social choice situation} is a structure $(A, \Ind, \Dom, \sigma)$ where:
\begin{itemize}
\item $A$ is a set of \emph{alternatives}.
\item $\Ind$ is a set of \emph{individuals} or \emph{agents}.
\item $\Dom \, \subseteq \, \Pref(A)^{\Ind}$ is the set of allowed \emph{ballots} or \emph{profiles} of individual preferences.
\item $\sigma : \Dom \rarr \Pref(A)$ is the \emph{social choice function}.
\end{itemize}

We write $p_i$ for the weak preference relation of the individual $i$ in a profile $p$. We write $\pstr_i$ for the strict preference relation associated with $p_i$. Similarly, we write $\sigstr$ for the strict preference relation associated with $\sigma(p)$.
We extend restriction to profiles pointwise: $(p | A)_i := p_i | A$. 

We shall now define a number of properties of social choice situations.
\begin{description}
\item[$\UD$] Unrestricted domain:
\[ \forall a, b, c \in A. \, \forall p \in \Pref(\{ a, b, c \})^{\Ind} . \, \exists q \in \Dom. \, q | \{ a, b, c \} = p . \]
\item[$\Pareto$] Pareto:
\[ \forall a, b \in A. \, \forall p \in D. \, (\forall i \in \Ind . \, a \pstr_i b ) \IMP a \sigstr b . \]
\item[$\WP$] Weak Pareto:
\[ \forall a, b \in A. \, \forall p \in D. \, (\forall i \in \Ind . \, a \pstr_i b ) \IMP a \sigma(p) b . \]
\item[$\IIA$] Independence of irrelevant alternatives:
\[ \forall a, b \in A. \, \forall p, q \in \Dom . \, p | \{ a, b \} = q | \{ a, b \} \IMP \sigma(p)  | \{ a, b \} = \sigma(q)  | \{ a, b \} . \]
\item[$\Dict$] Dictator:
\[ \exists i \in \Ind . \, \forall a, b \in A. \, \forall p \in \Dom. \, a \pstr_i b \IMP a \str{\sigma(p)} b . \]
\end{description}

We can now state Arrow's theorem.

\begin{theorem}[Arrow]
Let $(A, \Ind, \Dom, \sigma)$ be a social choice situation with $\card{A} \geq 3$, satisfying $\UD$, $\IIA$ and $\Pareto$. Then if  $\Ind$ is finite it also satisfies $\Dict$, \ie there is a dictator.
\end{theorem}

\subsection{Proof of Arrow's Theorem}

In this section we shall fix a social choice situation $(A, \Ind, \Dom, \sigma)$ satisfying the following conditions: $\card{A} \geq 3$, $\UD$, and $\IIA$.

\subsubsection{Decisiveness, Neutrality and Monotonicity}

In this subsection, we shall assume that our social choice situation satisfies the weak Pareto principle $\WP$.

Given $U \subseteq \Ind$ and $p \in \Dom$, we introduce the notation $a \str{p_U} b := \forall i \in U. \, a \pstr_i b$.
Given a set $U \subseteq \Ind$ and distinct elements $a, b \in A$, we define 
\[ U_{ab} \; := \; \{ p \in \Dom \mid  a \str{p_U} b \AND b \str{p_{U^{c}}} a  \} . \]
We define a relation $D_U$ on $A$ by 
\[ a D_U b \; := \; a \neq b \AND \forall p \in U_{ab}. \,  a \sigstr b . \]
We read $a D_U b$ as ``$U$ is decisive for $a$ over $b$''.

\begin{proposition}
\label{arrprop}
For all $a, b, c \in A$:
\begin{enumerate}
\item If $c \neq a$, then $a D_U b \IMP a D_U c$.
\item If $c  \neq b$, then $a D_U b \IMP c D_U b$.
\end{enumerate}
\end{proposition}
\begin{proof}
For (1), if $b=c$ there is nothing to prove. If $b \neq c$, we consider a profile $p \in \Dom$ such that for all $i \in U$, $p_i$ restricts to the strict chain $abc$, and for all $i \in \UC$, $p_i$ restricts to the strict chain $bca$. Such a profile exists by $\UD$. 
Note that $p \in U_{ab} \cap U_{ac} \cap \Ind_{bc}$. Since $a D_U b$, $a \sigstr b$, while by $\WP$, $b \sigma(p) c$. By transitivity if $b \sigstr c$, or by absorption otherwise, $a \sigstr c$. Now consider any profile $q$ such that $q \in U_{ac}$. Then $q | \{ a, c \} = p | \{a, c \}$, and by $\IIA$, $a \sigstr c \IMP a \sigma(q)^{>} c$; thus $a D_U c$. The argument for (2) is similar.
\end{proof}

As pointed out in \cite{blau1972direct}, the following purely relational argument allows us to conclude Neutrality from the previous proposition.

\begin{proposition}
\label{relprop}
Let $R$ be an irreflexive relation on a set $X$ with at least three elements, such that, for all $a, b, x \in X$:
\begin{enumerate}
\item If $x \neq a$, then $a R b \IMP a R x$.
\item If $x  \neq b$, then $a R b \IMP x R b$.
\end{enumerate}
If $x, y$ are any pair of distinct elements of $X$, then $a R b \IMP x R y$.
\end{proposition}
\begin{proof}
If $y \neq a$, then $a R b \IMP a R y \IMP x R y$.
If $x \neq b$, then $a R b \IMP x R b \IMP x R y$.
Otherwise, $x = b$ and $y = a$, and we must prove $a R b \IMP b R a$. In this case, since $X$ has at least three elements, we can find $c \in X$ with $a \neq c \neq b$. Then:
\[ a R b \IMP a R c \IMP b R c \IMP b R a . \]
\end{proof}

As an immediate consequence of Propositions~\ref{arrprop} and~\ref{relprop}, we obtain
\begin{theorem}[Local Neutrality]
\label{neutthm}
For all $a, b, x, y \in A$ with $x \neq y$:
\[ a D_U b \IMP x D_U y . \]
\end{theorem}



We now define a relation $E_U$ on $A$ by:
\[ a E_U b \; := \;  \forall p \in \Dom. \,  a \str{p_U} b \IMP a \sigma(p)^{>} b . \]
Thus we ask only that the individuals in $U$ strictly prefer $a$ to $b$; there is no constraint on those outside $U$.
Clearly, $a E_U b \IMP a D_U b$. The converse is an important property known as \emph{monotonicity}.

\begin{proposition}[Monotonicity]
\label{monprop}
For all $a, b \in A$, $a D_U b \IFF a E_U b$.
\end{proposition}
\begin{proof}
We shall prove $a D_U b \IMP a E_U b$.
Suppose we are given a profile $p$ such that $a \str{p_U} b$.
We can find an element $c \in A$ with $a \neq c \neq b$. We consider a profile $q \in \Dom$ such that for all $i \in U$, $q_i$ restricts to the strict chain $acb$, and for all $i \in \UC$, $c$ is strictly preferred to both $a$ and $b$ in $q_i$, while $q_i | \{ a, b\} = p_i | \{ a, b \}$. Such a profile exists by $\UD$. 
Note that $q \in U_{ac} \cap \Ind_{cb}$, and $q | \{ a, b \} = p | \{ a, b \}$. Since $a D_U b$, by Proposition~\ref{arrprop} $a D_U c$, and so $a \str{\sigma(q)} c$. By $\WP$, $c \sigma(q) b$. By transitivity if $c \str{\sigma(q)} b$, or by absorption otherwise, $a \str{\sigma(q)} b$. Since $p | \{ a, b \} = q | \{ a, b\}$, by $\IIA$ we conclude that $a \sigstr b$, and hence $a E_U b$ as required.
\end{proof}

\subsubsection{The Ultrafilter of Decisive Sets}

In this subsection, we assume 
the strong Pareto principle $\Pareto$, which serves as a basic existence principle for decisive sets. Note indeed that $\Pareto$ is equivalent to the statement that $\Ind$ is a decisive set.

We define
\[ \UU \; := \; \{ U \subseteq \Ind \mid \exists a, b \in A. \, a D_U b \} . \]

\begin{theorem}[The Ultrafilter Theorem]
\label{ultrafilthm}
$\UU$ is an ultrafilter.
\end{theorem}
\begin{proof}
(F1) As we have already noted, $\Pareto$ implies that $\Ind \in \UU$.

(F2) Now suppose that $U \in \UU$ and $U \subseteq V$. By Proposition~\ref{monprop}, we can conclude that $V \in \UU$, since clearly $U \subseteq V$ implies that $E_U \subseteq E_V$.

(F3) Now suppose for a contradiction that $U$ and $V$ are both in $\UU$, where $U \cap V = \varnothing$.
Consider a profile $p \in \Dom$ such that $a \str{p_U} b$ and $b \str{p_{V}} a$. By Proposition~\ref{monprop}, we have both $a \sigstr b$ and $b \sigstr a$, yielding a contradiction.

(F4) Finally suppose that $U \in \UU$ can be written as a disjoint union $U = V \sqcup W$.
We shall show that either $V \in \UU$ or $W \in \UU$.
Consider a profile $p \in \Dom$ such that for each $i \in V$, $p_i$ restricts to the strict chain $bca$, for each $i \in W$, $p_i$ restricts to the strict chain $cab$, while for each $i \in U^c$, $p_i$ restricts to the strict chain $abc$. Such a profile exists by $\UD$.
Note that $p \in U_{ca} \cap V_{ba} \cap W_{cb}$.
We argue by  cases:
\begin{itemize}
\item If $c \sigstr b$, then by $\IIA$, for all $q \in W_{cb}$, $c \str{\sigma(q)} b$, and hence $c D_W b$, and $W \in \UU$.
\item Otherwise, we must have $b \sigma(p) c$. Since $p \in U_{ca}$ and $U \in \UU$, using the Neutrality Theorem~\ref{neutthm}, we must have $c \sigstr a$. By absorption, $b \sigstr a$. By $\IIA$, for all $q \in V_{ba}$, $b \str{\sigma(q)} a$, and hence $b D_V a$, and $V \in \UU$.
\end{itemize}

The conditions (F1)--(F4) are easily seen to be equivalent to the standard definition of an ultrafilter, given as (F1) and (F2) together with:

(F5) $\varnothing \not\in \UU$.

(F6) $U, V \in \UU \IMP U \cap V \in \UU$.

(F7) $\forall U \subseteq \Ind. \, U \in \UU \OR U^c \in \UU$.
\end{proof}

We define the set of ballots which are linear on the alternatives $a, b$:
\[ \LL_{ab} \; := \;  \{ p \in \Dom \mid p | \{ a, b \} \in \Lin(\{ a, b \}) \} \; = \; \{ p \in \Dom \mid \exists U \subseteq \Ind. \, p \in U_{ab} \} . \]
We now show that $\UU$ completely determines $\sigma$ on linear ballots.
\begin{proposition}
\label{Udetswprop}
For all $a, b \in A$, $p \in \LL_{ab}$:
\[ a \sigstr b \IFF \{ i \in \Ind \mid a \str{p_i} b \} \in \UU . \]
\end{proposition}
\begin{proof}
The right-to-left implication is immediate, since if $U =  \{ i \in \Ind \mid a \str{p_i} b \}$, $p \in \LL_{ab}$ implies that $p \in U_{ab}$.
For the converse, we use property (F7) from Theorem~\ref{ultrafilthm}.
\end{proof}

We also show that social choice functions map linear ballots to strict preferences.
\begin{proposition}
\label{linresprop}
For all distinct alternatives $a, b \in A$:
\[ \forall p \in \LL_{ab}. \, a \sigstr b \OR b \sigstr a . \]
\end{proposition}
\begin{proof}
Immediate from the previous proposition and property (F7) from Theorem~\ref{ultrafilthm}.
\end{proof}

\subsection{Arrow's Theorem}

\begin{theorem}[Arrow]
Let $(A, \Ind, \Dom, \sigma)$ be a social choice situation with $\card{A} \geq 3$, satisfying $\UD$, $\IIA$ and $\Pareto$. Then if  $\Ind$ is finite it also satisfies $\Dict$, \ie there is a dictator.
\end{theorem}
\begin{proof}
By the Ultrafilter Theorem~\ref{ultrafilthm}, $\UU$ is an ultrafilter. Since $\Ind$ is finite, $\UU$ must be principal, consisting of all supersets of $\{ i \}$ for some $i \in \Ind$. Then $D_{\{ i \}}$ is decisive, or equivalently by Proposition~\ref{monprop},  $i$ is a dictator.
\end{proof}

\section{Categorical Formulation of Arrow's Theorem}

Given a universe $\AU$ of possible alternatives, where the cardinality of $\AU$ is $\geq 3$, we consider the category $\CC$ whose objects are subsets of $\mathbb{A}$, and whose morphism are injective maps; and its posetal sub-category $\Cinc$ with morphisms the inclusions.
We shall  use  $\CC^{(k)}$ and $\Cinc^{(k)}$ to denote the full sub-categories of $\CC$ and $\Cinc$ respectively determined by the sets $A$ of cardinality $\leq k$, for $k \geq 0$. We write $\CC^{\op}$ for the opposite category of $\CC$.

For any notion of binary preference relation axiomatized by \emph{universal} sentences (universal closures of quantifier-free formulas), we get a functor
\[ \PP : \CC^{\op} \lrarr \Set . \]
$\PP(X)$ is the set of preference relations on $X$, and if $f : X \rinj Y$ and $p \in \PP(Y)$, then we define
\[ x \, (\PP(f)(p)) \, x'  \IFF f(x) \, p \, f(x') . \]
Note that injectivity ensures that $(X, \PP(f)(p))$ is isomorphic to a sub-structure of $(Y, p)$, and hence the truth of universal sentences is preserved \cite{hodges1997shorter}.

Also note that $\PP$ cuts down to a functor $\PP_{\mathsf{inc}} : \Cinc^{\op} \lrarr \Set$.

We shall use  $\Pref$ to denote the functor induced by the notion of weak preference relation introduced in the previous section, and $\Lin$ for the subfunctor of linear preference relations.

\subsection{Categorical formulation of  $\UD$}

The functor $\PI$ is defined as the product of $\Ind$ copies of $\Pinc$; thus for each $A$, $\PI(A) := \Pinc(A)^{\Ind}$. This gives the set of all possible profiles over a set of alternatives $A$ for the agents in $\Ind$.
We shall assume we are given a subfunctor $\Dinj$ of $\Pref^{\Ind}$. Thus $\Dinj : \CC^{\op} \lrarr \Set$ is a functor, with a natural transformation $\Dinj \natarrow \Pref^{\Ind}$ whose components are inclusion maps.
$\Dinj$ restricts to a functor $\Dinc : \Cinc^{\op} \lrarr \Set$.

The axiom $\UD$ can be stated in these terms as follows:

\begin{center}
\begin{tabular}{|lll|}\hline
\textbf{(CUD)} & (i) & For $A \in \CC^{(3)}$, $\Dinc(A) = \PI(A)$. \\
& (ii) & $\Dinc$ preserves epis. \\ \hline
\end{tabular}
\end{center}
The requirement that $\Dinc$ preserves epis means that  inclusions $\iota : A \rinc B$ are mapped to surjections $\Dinc(\iota) : \Dinc(B) \lrarr \Dinc(A)$.


\subsection{Categorical formulation of $\IIA$}
 
 Next, we make the observation that 
\textbf{IIA} is equivalent to the following statement:
\begin{center}
\begin{tabular}{|lc|}\hline
\textbf{(CIIA)} & The social welfare function is a natural transformation \\ 
& $\sw : \Dinc \natarrow \Pinc$. \\ \hline
\end{tabular}
\end{center}
Explicitly, this says that for each set of alternatives $A$ we have a map
\[ \sw_A : \Dinc(A) \lrarr \Pinc(A) \]
such that, for all inclusions $A \subseteq B$ and profiles $p \in \Dinc(A)$:
\[ \sw_A(p | A) = \sw_{B}(p) | A  
\]
More precisely, we have the following result.
\begin{proposition}
We assume that $\Dinc$ is a subfunctor of $\PI$ satisfying $\CUD$.
\begin{enumerate}
\item If $\sw_{\AU} : \Dinc(\AU) \lrarr \Pinc(\AU)$ is a function satisfying $\IIA$, then it extends to a natural transformation $\sw : \Dinc \natarrow \Pinc$.
\item If $\sw : \Dinc \natarrow \Pinc$ is a natural transformation, then for every $A \in \CC$,
$\sw_A$ satisfies $\IIA$.
\end{enumerate}
\end{proposition}
\begin{proof}
1. Firstly, note that $(\AU, \Ind, \Dinc(\AU), \sw_{\AU})$ is a social choice situation in the sense of the previous section. We are assuming that this structure satisfies 
$\IIA$.

We note that $\IIA$ implies the following, more general statement: for all $A \subseteq \AU$, and $p, q \in \Dinc(\AU)$,
\[ p | A = q | A \IMP \sw_{\AU}(p) | A = \sw_{\AU}(q) | A . \]
This holds because any binary relation on a set $X$  is determined by its restrictions to the subsets of $X$ of cardinality $\leq 2$.

Given $A \subseteq \AU$ and $p \in \Dinc(A)$, by $\CUD$ there is $q \in \Dinc(\AU)$ such that $q | A = p$. We define $\sw_A(p) := \sw_{\AU}(q) | A$. By our previous remark, this is independent of the choice of $q$. For naturality, if $\iota : A \rinc B$ and $p \in \Dinc(B)$, then for any $q \in \Dinc(\AU)$ such that $q | B = p$, $q | A = \Dinc(\iota)(q | B)$, and hence 
\[ \sw_A \circ \Dinc(\iota)(p) = \sw_{\AU}(q) | A = (\sw_{\AU}(q) |B) | A = \Pinc(\iota) \circ \sw_B(p) . \]

2. For the converse, if $\sw : \Dinc \natarrow \Pinc$ is a natural transformation, $A \subseteq \AU$, and $ \iota : \{a, b \} \rinc A$, then $p, q \in \Dinc(A)$ with $p | \{ a, b\} = q | \{ a, b \}$ means that $\Dinc(\iota)(p) = \Dinc(\iota)(q)$. Using naturality, we have
\[ \sw_A (p) | \{ a, b\} = \Pinc(\iota) \circ \sw_A (p) = \sw_{ \{ a, b\} } \circ \Dinc(\iota) (p) = \sw_{ \{ a, b\} } \circ \Dinc(\iota) (q) = \sw_A(q) | \{ a, b\} . 
\]
\end{proof}

\subsection{Categorical formulation of $\Pareto$}
Consider the standard diagonal map $\Delta_I : X \lrarr X^{\Ind}$. An arrow $f : X^{\Ind} \rarr X$ is \emph{diagonal-preserving} if $f \circ \Delta_I = \id_{X}$.
The Pareto condition is essentially a form of diagonal preservation.

Firstly, note that given a functor $F : \Cinc^{\op} \lrarr \Set$, we can define the restriction $F^{(2)} : (\Cinc^{(2)})^{\op} \lrarr \Set$. Now $\Delta_I$ induces a natural transformation 
$\Lininc^{(2)} \natarrow (\PI)^{(2)}$. Using part (i) of $\CUD$, this factors through the inclusion $\Dinc^{(2)} \rinc (\PI)^{(2)}$. Thus we obtain a natural transformation  $\Delta_I : \Lininc^{(2)} \natarrow \Dinc^{(2)}$.
Also, $\Lininc^{(2)}$ is a sub-functor of $\Pinc^{(2)}$, with inclusion $e : \Lininc^{(2)} \natarrow \Pinc^{(2)}$. The categorical formulation of the Pareto condition is now as follows:
\begin{center}
\begin{tabular}{|lc|}\hline
\textbf{(CP)} & $\sigma \circ \Delta_I = e $. \\ \hline
\end{tabular}
\end{center}
Diagrammatically, this is
\begin{diagram}[4em]
\Lininc(\{ a, b \}) & \rTo^{\Delta_I} & \Dinc(\{a, b \}) \\
& \rdinc & \dTo_{\sw_{\{ a, b \}} } \\
& & \Pinc(\{ a, b \})
\end{diagram}

\begin{proposition}
Let $\Dinc$  be a subfunctor of $\PI$ satisfying $\CUD$, and $\sw : \Dinc \natarrow \Pinc$ a natural transformation. Then $\sw$ satisfies $\CP$ if and only if for every $A \subseteq \AU$, $(A, \Ind, \Dinc(A), \sw_A)$ satisfies $\Pareto$.
\end{proposition}

\subsection{Categorical Formulation of Arrow's Theorem}

A \emph{categorical social choice situation} is given by a set $\AU$ determining a category $\CC$, a set $\Ind$ of individuals, a subfunctor $\Dinc$ of $\PI$ satisfying $\CUD$, and a natural transformation $\sw : \Dinc \natarrow \Pinc$.

\begin{theorem}[Arrow's Theorem: Categorical Statement]
\label{catarrowthm}
Let $(\AU, \Ind, \Dom, \sw)$ be a categorical social choice situation where $\card{\AU} \geq 3$, $\Ind$ is finite, and $\sw$ satisfies $\CP$. Then $\sw = \pi_i$ for some fixed $i \in \Ind$, where $(\pi_i)_A : p \mapsto p_i$. 
\end{theorem}
More colloquially, this can be stated as:
\begin{center}
\fbox{The only diagonal-preserving natural transformations $\sw : \Dinc \natarrow \Pinc$
are the projections.}
\end{center}


\subsection{An Analogous Result in Type Theory}

We remark that when Arrow's theorem is formulated in this way, it displays an evident kinship with a well-studied genre of results in functional programming and type theory \cite{wadler1989theorems,bainbridge1990functorial}.
These results
use (di)naturality constraints to show that the behaviour of polymorphic terms are essentially determined by their types.

We illustrate these ideas with an example.

\paragraph{Question}
What natural transformations
\[ t_X : X^2 \natarrow X \]
can there be in the functor category $[\Set, \Set]$?

\paragraph{Answer}
The only such natural transformations are the projections.

\paragraph{Sketch}
Use naturality to show this first for two-element sets, then to lift it. E.g.

\[
\begin{diagram}
\{ a, b \}^2 & \rTo^{\pi_1} & \{ a, b \} \\
\dTo & & \dTo \\
X^2 & \rTo_{t_X} & X
\end{diagram}
\qquad \qquad
\begin{diagram}
 (a, b) & \rmto^{\pi_1} & a  \\
  \dmto & & \dmto \\
  (x, y) & \rmto_{t_X} & x 
\end{diagram}
\]

\paragraph{Exercise} Show that the same result holds for natural transformations
$X^{\Ind} \natarrow X$.

\section{Categorical perspective on the proof of Arrow's Theorem}

We shall now revisit the proof of Arrow's theorem from the categorical perspective.

We shall assume throughout this section that we are given a categorical social choice situation  $(\AU, \Ind, \Dom, \sw)$ satisfying $\CP$.
Note firstly that, for each $A \subseteq \AU$, $(A, \Ind, \Dom(A), \sw_A)$ is a standard social choice situation satisfying $\UD$, $\IIA$ and $\Pareto$.
We shall use results from Section~\ref{concsec} freely.

\subsection{Neutrality as Naturality}

The property of Neutrality, which in the concrete setting was stated in a `local' form in Theorem~\ref{neutthm}, becomes a form of naturality.
To state this properly, we need to consider the subfunctor $\DL$ of $\Dinj$, where $\DL(A) = \Dinj(A) \cap \Lin(A)^{\Ind}$. Thus $\DL(A)$ is the set of admissible \emph{linear} ballots. By Proposition~\ref{linresprop}, $\sw$ cuts down to a natural transformation $\swl : \DLinc \natarrow \Lininc$.

The key neutrality property becomes the following:

\begin{proposition}[Neutrality: categorical version]
The social choice map $\swl$ extends to a natural transformation $\swl : \DL \natarrow \Lin$.
\end{proposition}

The assertion of this proposition is that a social welfare function is natural with respect, not just to inclusions, but to injective maps.
Note that the family of maps $\{ \swl_A \}$ is the same: we are claiming that additional naturality squares commute.

\begin{proof}

Let $\alpha : A \rinj B$ be an injective map. We must show that $f = g$, where
\[ f := \swl_A \circ \DL(\alpha), \qquad g := \Lin(\alpha) \circ \swl_B . \]
Note that, for $r \in \Lin(B)$ and $a, a' \in A$,
 \[ a\, (\Lin(\alpha)(r)) \, a' \IFF \alpha(a) \, r \, \alpha(a') , \] 
and  for $p \in \DL(B)$, 
\[ (\DL(\alpha)(p))_i \; = \; \Lin(\alpha)(p_i) . \]
For $p \in \DL(B)$, and distinct $a, a' \in A$, by Proposition~\ref{Udetswprop}:
\[ \begin{array}{lcl}
a \, f(p) \, a' & \IFF & \{ i \in \Ind \mid a \, (\DL(\alpha)(p))_i \, a' \} \in \UU \\
& \IFF  & \{ i \in \Ind \mid \alpha(a) \, p_i \, \alpha(a') \} \in \UU \\
&  \IFF & \alpha(a) \, \swl_B(p) \, \alpha(a') \\
& \IFF & a \, g(p) \, a' . 
\end{array}
\]
\end{proof}

\subsection{The Factorization Theorem}

We now recast the Ultrafilter Theorem into an arrow-theoretic form. Firstly, we recall some standard notions from boolean algebra \cite{sikorski1969boolean}.

\begin{proposition}
\label{boolalgprop}
Given a set $\Ind$, there is a bijection by characteristic functions between families $\UU$ of subsets of $\Ind$, and functions $h : \TI \lrarr \Two$, where $\Two := \{ 0, 1 \}$.
A family $\UU$ is a ultrafilter if and only if the corresponding function $h$ is a boolean algebra homomorphism.
If $\Ind$ is finite, the only such homomorphisms are the projections.
\end{proposition}

We can define maps
\[ \phi : \Lin(A) \lrarr \Two^{A^2} \; :: \; \phi(r)(a, a') = 1 \IFF a \, r^{>} \, a' \]
and
\[ \psi : \DL(S) \lrarr (\TI)^{A^2} \; :: \; \psi(p)(a, a')_i = 1 \IFF a \, p_i^{>} \, a' .\]
The following is a restatement in arrow-theoretic terms of the Ultrafilter Theorem.

\begin{theorem}[Factorization Theorem]
For any social choice function $\sw : \DL \natarrow \Lin$, there is a boolean algebra homomorphism
\[ h : 2^{\Ind} \lrarr 2 \]
such that the following diagram commutes:
\[ \begin{diagram}[4em]
\DL(A) & \rTo^{\sigma_A} & \Lin(A) \\
\dTo^{\psi} & & \dTo_{\phi} \\
 (\TI)^{A^2} & \rTo_{h^{A^2}} & \Two^{A^2}
\end{diagram} \]
\end{theorem}

The content of this result is that all the information needed to compute the social welfare function $\sw$ is contained in the boolean algebra homomomorphism $h$.

The categorical form of Arrow's theorem, Theorem~\ref{catarrowthm}, follows immediately from the Factorization Theorem and the last remark in Proposition~\ref{boolalgprop}.

\section{Discussion}

One of our motivations in undertaking this study of Arrow's theorem was to see if common structure could be identified with notions such as no-signalling, parameter independence etc. which play a key r\^ole in quantum foundations. Arrow's theorem is a no-go theorem of a similar flavour to results such as Bell's theorem.
A central assumption is $\IIA$, which is analogous to the various forms of independence which appear as hypotheses of the results in quantum foundations.
In particular, the functorial treatment we have developed in the present paper has common features with the r\^ole of presheaves in the sheaf-theoretic account of quantum non-locality and contextuality given in \cite{abramsky2011sheaf}.

It must be said that, although some degree of commonality has been exposed by the present account, the arguments are substantially different. Nevertheless, the use of the categorical language to put results from such different settings in a common framework is suggestive, and may prove fruitful in exploring the r\^ole of various forms of independence.
It will also be interesting to relate this to the logics of dependence and independence being developed by Jouko V\"a\"an\"anen and his colleagues.

Altogether, although modest in its scope, we hope the present paper may help to suggest some further possibilities for elucidating the general structure of no-go results, and of the notions of independence which play a pervasive part in these results.

\bibliographystyle{alpha}
\bibliography{bibfile}

\begin{thebibliography}{BFSS90}

\bibitem[AB11]{abramsky2011sheaf}
S.~Abramsky and A.~Brandenburger.
\newblock The sheaf-theoretic structure of non-locality and contextuality.
\newblock {\em New Journal of Physics}, 13(11):113036, 2011.

\bibitem[Arr50]{arrow1950difficulty}
K.J. Arrow.
\newblock A difficulty in the concept of social welfare.
\newblock {\em The Journal of Political Economy}, 58(4):328--346, 1950.

\bibitem[Arr63]{arrow1970social}
K.J. Arrow.
\newblock {\em Social choice and individual values}.
\newblock Yale University Press, 1963.
\newblock Second edition.

\bibitem[BFSS90]{bainbridge1990functorial}
E.S. Bainbridge, P.J. Freyd, A.~Scedrov, and P.J. Scott.
\newblock Functorial polymorphism.
\newblock {\em Theoretical Computer Science}, 70(1):35--64, 1990.

\bibitem[Bla72]{blau1972direct}
J.H. Blau.
\newblock {A direct proof of Arrow's theorem}.
\newblock {\em Econometrica}, 40:61--67, 1972.

\bibitem[Gib73]{gibbard1973manipulation}
A.~Gibbard.
\newblock Manipulation of voting schemes: a general result.
\newblock {\em Econometrica: Journal of the Econometric Society}, pages
  587--601, 1973.

\bibitem[GV12]{gradeldependence}
E.~Gr{\"a}del and J.~V{\"a}{\"a}n{\"a}nen.
\newblock {Dependence, Independence, and Incomplete Information}.
\newblock In {\em Proceedings of 15th International Conference on Database
  Theory, ICDT 2012}. ACM, 2012.

\bibitem[Hod97]{hodges1997shorter}
W.~Hodges.
\newblock {\em A shorter model theory}.
\newblock Cambridge University Press, 1997.

\bibitem[KS72]{kirman2012arrow}
A.P. Kirman and D.~Sondermann.
\newblock Arrow's theorem, many agents, and invisible dictators.
\newblock {\em Journal of Economic Theory}, 5(2):267--277, 1972.

\bibitem[Sat75]{satterthwaite1975strategy}
M.A. Satterthwaite.
\newblock Strategy-proofness and arrow's conditions: Existence and
  correspondence theorems for voting procedures and social welfare functions.
\newblock {\em Journal of economic theory}, 10(2):187--217, 1975.

\bibitem[Sik69]{sikorski1969boolean}
R.~Sikorski.
\newblock {\em Boolean algebras}.
\newblock Springer, 1969.

\bibitem[Tay05]{taylor2005social}
A.D. Taylor.
\newblock {\em Social choice and the mathematics of manipulation}, volume~6.
\newblock Cambridge University Press, 2005.

\bibitem[V{\"a}{\"a}07]{vaananen2007dependence}
J.~V{\"a}{\"a}n{\"a}nen.
\newblock {\em Dependence logic: a new approach to independence friendly
  logic}, volume~70.
\newblock Cambridge University Press, 2007.

\bibitem[Wad89]{wadler1989theorems}
P.~Wadler.
\newblock Theorems for free!
\newblock In {\em Proceedings of the fourth international conference on
  Functional programming languages and computer architecture}, pages 347--359.
  ACM, 1989.

\end{thebibliography}

\end{document}